\documentclass[11pt]{amsart}

\usepackage[T1]{fontenc}
\usepackage[utf8]{inputenc}
\usepackage{lmodern}
\usepackage{amsmath,amssymb,amsthm,mathtools}
\usepackage{enumitem}
\usepackage{booktabs}
\usepackage{array}
\usepackage{hyperref}
\usepackage[a4paper,margin=1.15in]{geometry}
\hypersetup{
  colorlinks=true,
  linkcolor=blue,
  citecolor=blue,
  urlcolor=blue
}

\newtheorem{theorem}{Theorem}[section]
\newtheorem{proposition}[theorem]{Proposition}
\newtheorem{lemma}[theorem]{Lemma}
\newtheorem{corollary}[theorem]{Corollary}
\newtheorem{problem}[theorem]{Problem}
\theoremstyle{definition}
\newtheorem{definition}[theorem]{Definition}
\newtheorem{example}[theorem]{Example}
\theoremstyle{remark}

\DeclareMathOperator{\Par}{Par}
\DeclareMathOperator{\Cl}{Cl}
\DeclareMathOperator{\supp}{supp}
\DeclareMathOperator{\adist}{adist}
\DeclareMathOperator{\trk}{trk}
\DeclareMathOperator{\sgn}{sgn}

\title[Jump and Gradient Invariants]{Jump and Gradient Invariants in the Partition Graph}
\author{Fedor B. Lyudogovskiy}

\begin{document}

\begin{abstract}
We introduce edgewise jump invariants and gradient-type structures for the partition graph \(G_n\): its vertices are the partitions of \(n\), and its edges correspond to elementary transfers of one unit between parts. Previous work on \(G_n\) has focused primarily on vertex-level invariants such as degree, local simplex dimension, and support size. In this paper we study the edgewise variation of such invariants. For an oriented edge \(e=(\lambda,\mu)\) and a vertex invariant \(F\), we define the signed jump \(\Delta_eF=F(\mu)-F(\lambda)\), and we focus on the basic jump signature
\[
J(e)=(\Delta_ed,\Delta_e\delta,\Delta_e\sigma),
\]
where \(d\) is degree, \(\delta\) is local simplex dimension, and \(\sigma\) is support size. We prove that support jumps are universally bounded by \(2\) and describe them in terms of local multiplicity data. We also develop a taxonomy of active, neutral, pure, and mixed transitions, relate nonzero jumps to threshold-layer crossings, and introduce strict gradient orientations associated with arbitrary real-valued vertex invariants. Finally, we formulate a computational atlas framework for studying jump spectra, transition ranks, localization of large jumps, and monotone corridors. No large-scale computations are carried out here; the atlas is presented as a reproducible protocol for subsequent work. This framework provides an edgewise language for comparing local vertex invariants with larger-scale structural patterns in \(G_n\).
\end{abstract}

\maketitle



\medskip

\section{Introduction}

The partition graph \(G_n\) has a simple definition but a rich local morphology. Its vertices are the integer partitions of \(n\), and two partitions are adjacent when one can be obtained from the other by moving one unit from one part to another part and then reordering. For standard background on partitions, see \cite{Andrews,Macdonald,Stanley}; for graph-theoretic and combinatorial-topological terminology, see \cite{Diestel,Kozlov}. This elementary transfer operation produces a finite graph whose local structure reflects several interacting features of a partition: its available transfers, support, degree, local clique structure, and position relative to global regions such as the boundary, the center, and the self-conjugate axis.

Previous work in this series has developed several vertex-level invariants and geometric structures of \(G_n\). The present paper uses only the parts of that framework most directly needed for edgewise transition language. These include the clique-complex background from \cite{Lyu2026Homotopy}, the local transfer and local morphology formalism from \cite{Lyu2026Local}, degree theory from \cite{Lyu2026DegreeTheory}, simplex shells and simplex-layer boundaries from \cite{Lyu2026SimplicialShells,Lyu2026SimplexLayers}, support terminology from \cite{Lyu2026Support}, and the broader geometric viewpoint from \cite{Lyu2026Growing,Lyu2026Directional}. Among the vertex-level invariants used below are the degree
\[
d(\lambda)=\deg_{G_n}(\lambda),
\]
the local simplex dimension
\[
\delta(\lambda)=\dim_{\mathrm{loc}}(\lambda)
\]
in the clique complex
\[
K_n=\Cl(G_n),
\]
and the support size
\[
\sigma(\lambda)=|\supp(\lambda)|.
\]
These invariants describe the local morphology of individual vertices. They measure, respectively, graph-theoretic branching, local simplicial thickness, and the number of distinct part sizes present in a partition.

Here we study a different but closely related question. Instead of asking only what value an invariant has at a vertex, we ask how that value changes along an edge. Thus, for an oriented edge
\[
e=(\lambda,\mu)
\]
and a vertex invariant \(F\), we consider the signed jump
\[
\Delta_eF=F(\mu)-F(\lambda).
\]
Passing from \(F(\lambda)\) to \(\Delta_eF\) shifts the emphasis from static vertex morphology to edgewise transition morphology. The main objects of the paper are jump invariants and gradient-type structures, which describe how local regimes change under one elementary transfer.

The role of this paper in the series is primarily structural: it establishes a common edgewise language and isolates a small set of elementary structural facts, namely the support-jump bound, the local multiplicity description of support jumps, the threshold-crossing interpretation of integer-valued jumps, and the acyclicity of strict gradient orientations. The computational atlas language introduced later is intended as a reproducible framework for subsequent calculations rather than as a complete classification of all realized jump signatures.

The basic jump signature studied here is
\[
J(e)=(\Delta_ed,\Delta_e\delta,\Delta_e\sigma).
\]
It records, for a single edge, the simultaneous change in degree, local simplex dimension, and support size. The absolute signature
\[
|J|(e)=(|\Delta_ed|,|\Delta_e\delta|,|\Delta_e\sigma|)
\]
is an invariant of the underlying unoriented edge.

This language makes it possible to classify elementary transfers according to whether they are degree-active, dimension-active, support-active, pure, mixed, or fully neutral. It also makes it possible to distinguish numerical neutrality from structural neutrality: an edge may preserve the values of \(d\), \(\delta\), and \(\sigma\), while still changing the neighbor system or local transfer structure.

A first exact result concerns support jumps. If
\[
\sigma(\lambda)=|\supp(\lambda)|,
\]
then every edge satisfies
\[
|\Delta_e\sigma|\le 2.
\]
Moreover, support jumps are determined by local multiplicity data. If an elementary transfer is written as
\[
p\mapsto p-1,\qquad q\mapsto q+1,
\]
then the support jump is determined by the transfer data \((p,q)\) together with the multiplicities of the affected sizes
\[
p,\quad q,\quad p-1,\quad q+1,
\]
where these expressions are understood as a set of affected sizes and some of them may coincide.
The support component of every jump signature therefore belongs to the fixed set
\[
\{-2,-1,0,1,2\},
\]
independently of \(n\).

Degree and local-dimension jumps are subtler. A degree jump compares the sizes of two neighbor systems:
\[
\Delta_ed=|N(\mu)|-|N(\lambda)|.
\]
Equivalently, it measures the imbalance between neighbors gained and neighbors lost when passing from \(\lambda\) to \(\mu\). Local-dimension jumps compare maximal local simplex structures through the two endpoints. These jumps are not controlled by the support argument alone and therefore require separate structural and computational analysis.

The paper also introduces a gradient-type viewpoint. Every real-valued invariant
\[
F:\Par(n)\to\mathbb R
\]
defines a strict \(F\)-gradient orientation of \(G_n\): an edge is oriented from \(\lambda\) to \(\mu\) whenever
\[
F(\mu)>F(\lambda).
\]
This orientation is acyclic. Thus each invariant gives its own notion of upward movement, plateau edges, monotone paths, and corridors. There is no single canonical gradient on \(G_n\); rather, degree, local dimension, support size, axial distance, and other invariants may define different and sometimes opposed directions of drift.

For integer-valued invariants, jumps are also equivalent to threshold-layer crossings. If
\[
L_F^{\ge r}(n)=\{\lambda:F(\lambda)\ge r\},
\]
then an edge has nonzero \(F\)-jump exactly when it crosses the boundary of at least one threshold layer. More precisely,
\[
|\Delta_eF|
\]
counts the number of integer threshold boundaries crossed by the edge. This gives a direct connection between jump morphology and layer morphology.

The final part of the paper formulates a protocol for a computational atlas of edgewise transitions. For fixed \(n\), one may compute the realized signed signatures
\[
\mathcal J(n)=\{J(e):e\in\mathcal E_n^{\mathrm{or}}\},
\]
the absolute signatures
\[
\mathcal J_{\mathrm{abs}}(n)=\{|J|(e):e\in E(G_n)\},
\]
jump histograms, transition-rank distributions, and localization data relative to \((a,b)\)-stacks, axial distance, degree layers, simplex layers, support strata, shells, and boundary regions.

A broader purpose of this framework, developed further in the series, is to prepare a bridge between two levels of the theory. On the one hand, the graph \(G_n\) has rich local and regional morphology. On the other hand, the associated clique complex \(K_n\) has comparatively simple global homotopy type in the existing theory. The edgewise framework developed here provides an intermediate language for studying how local regimes meet, how threshold and layer boundaries are crossed, and how transition patterns are organized across the graph.

The paper is organized as follows. Section~\ref{sec:edgewise-language} introduces oriented edges, signed and absolute jumps, jump signatures, and strict gradient orientations. Section~\ref{sec:support-jumps} proves the basic support-jump bound and describes support jumps in terms of local multiplicity data. Section~\ref{sec:degree-dimension-jumps} discusses degree and local-dimension jumps in terms of neighbor-system and simplex-system reorganization. Section~\ref{sec:joint-profiles} develops the taxonomy of pure, mixed, active, and neutral transitions. Section~\ref{sec:gradient-corridors} introduces gradient-type behavior, monotone paths, and corridors. Section~\ref{sec:atlas} describes the computational atlas framework for localizing jump regimes. Section~\ref{sec:conclusion} concludes with open problems and the role of edgewise morphology in the broader structure of the partition-graph program.

\section{Edgewise transition language}\label{sec:edgewise-language}

Let \(n\ge 1\). We write \(\Par(n)\) for the set of integer partitions of \(n\), and we denote by \(G_n\) the partition graph on \(\Par(n)\). Thus two partitions are adjacent if one can be obtained from the other by moving one unit from one part to another part and then reordering the resulting parts.

We now fix notation for variations of vertex invariants along edges.

\subsection{Oriented edges and jumps}

Although \(G_n\) is undirected, we shall often use an auxiliary orientation of an edge. We write
\[
e=(\lambda,\mu)
\]
for the oriented edge from \(\lambda\) to \(\mu\), where \(\lambda\sim\mu\) in \(G_n\). The set of oriented edges will be denoted by
\[
E^{\mathrm{or}}(G_n)=\{(\lambda,\mu):\{\lambda,\mu\}\in E(G_n)\}.
\]
The same unoriented edge with the opposite orientation will be denoted by
\[
\bar e=(\mu,\lambda).
\]

Let
\[
F:\Par(n)\to A
\]
be a vertex invariant with values in an abelian group \(A\). For an oriented edge \(e=(\lambda,\mu)\), define the signed jump of \(F\) along \(e\) by
\[
\Delta_eF=F(\mu)-F(\lambda).
\]
When \(A=\mathbb Z\) or \(A=\mathbb R\), we also define the absolute jump magnitude
\[
|\Delta_eF|=|F(\mu)-F(\lambda)|.
\]

The signed jump depends on the orientation of the edge, while the absolute jump does not.

\begin{lemma}[Orientation reversal]
Let \(F:\Par(n)\to A\) be a vertex invariant with values in an abelian group. If \(e=(\lambda,\mu)\) and \(\bar e=(\mu,\lambda)\), then
\[
\Delta_{\bar e}F=-\Delta_eF.
\]
In particular, if \(F\) is real-valued or integer-valued, then
\[
|\Delta_{\bar e}F|=|\Delta_eF|.
\]
\end{lemma}

\begin{proof}
By definition,
\[
\Delta_{\bar e}F=F(\lambda)-F(\mu)=-\bigl(F(\mu)-F(\lambda)\bigr)=-\Delta_eF.
\]
Taking absolute values gives the second assertion.
\end{proof}

Thus signed jumps are naturally attached to oriented edges, whereas absolute jumps are naturally attached to unoriented edges.

\subsection{The main vertex invariants}

The main vertex invariants considered in this paper are the degree, the local simplex dimension, and the support size.

First,
\[
d(\lambda)=\deg_{G_n}(\lambda)
\]
denotes the degree of \(\lambda\) in the partition graph.

Second, let
\[
K_n=\Cl(G_n)
\]
be the clique complex of \(G_n\). The local simplex dimension of \(\lambda\) is
\[
\delta(\lambda)=\dim_{\mathrm{loc}}(\lambda),
\]
defined as the maximum dimension of a simplex of \(K_n\) containing \(\lambda\).

Third, write
\[
\lambda=(1^{m_1}2^{m_2}3^{m_3}\cdots)
\]
in multiplicity notation. The support of \(\lambda\) is
\[
\supp(\lambda)=\{i\ge 1:m_i(\lambda)>0\},
\]
and its support size is
\[
\sigma(\lambda)=|\supp(\lambda)|.
\]

For an oriented edge \(e=(\lambda,\mu)\), the corresponding basic jumps are
\[
\Delta_ed=d(\mu)-d(\lambda),
\]
\[
\Delta_e\delta=\delta(\mu)-\delta(\lambda),
\]
and
\[
\Delta_e\sigma=\sigma(\mu)-\sigma(\lambda).
\]

These three quantities measure different aspects of an elementary transition. The degree jump records the change in graph-theoretic local branching. The local-dimension jump records the change in local simplicial thickness. The support jump records the change in the number of distinct part sizes.

\subsection{Jump signatures}
\label{subsec:jump-signatures}

The signed jump signature of an oriented edge \(e=(\lambda,\mu)\) is defined by
\[
J(e)=(\Delta_ed,\Delta_e\delta,\Delta_e\sigma).
\]
The absolute jump signature of the corresponding unoriented edge is
\[
|J|(e)=(|\Delta_ed|,|\Delta_e\delta|,|\Delta_e\sigma|).
\]

The signed signature records direction. The absolute signature records magnitude.

By the orientation-reversal lemma,
\[
J(\bar e)=-J(e),
\]
whereas
\[
|J|(\bar e)=|J|(e).
\]

One may also consider expanded jump signatures containing additional geometric or positional invariants. For example, if
\[
a(\lambda)=\lambda_1
\]
is the largest part and
\[
b(\lambda)=\ell(\lambda)
\]
is the number of parts, then one may consider
\[
J_{\mathrm{ext}}(e)=(\Delta_ed,\Delta_e\delta,\Delta_e\sigma,\Delta_ea,\Delta_eb,\Delta_e|a-b|).
\]
In this paper, however, the basic signature \(J(e)\) will be the primary object. Positional components such as \(\Delta a\), \(\Delta b\), and \(\Delta |a-b|\) will be used only when discussing localization, axial drift, or computational atlases.

\subsection{Jump classes and transition zones}

For a vertex invariant \(F\), define the oriented nonzero jump edge set by
\[
E_{F,\mathrm{or}}^{\neq 0}(n)=\{e\in E^{\mathrm{or}}(G_n):\Delta_eF\neq 0\}.
\]
For \(r\ge 0\), define the unoriented large absolute jump edge set by
\[
E_{|F|}^{\ge r}(n)=\{\{\lambda,\mu\}\in E(G_n):|F(\mu)-F(\lambda)|\ge r\}.
\]

Thus signed jump classes are naturally defined on oriented edges, while absolute jump classes are naturally defined on unoriented edges.

For several invariants \(F_1,\dots,F_k\), one obtains oriented joint jump classes
\[
E^{\mathbf v}_{F_1,\dots,F_k}(n)=\{e\in E^{\mathrm{or}}(G_n):(\Delta_eF_1,\dots,\Delta_eF_k)=\mathbf v\}.
\]
In particular, for the basic signature \(J(e)\), define
\[
E_{\mathbf v}^{\mathrm{or}}(n)=\{e\in E^{\mathrm{or}}(G_n):J(e)=\mathbf v\}.
\]

These sets provide the edgewise analogue of vertex strata. Instead of grouping vertices by the value of an invariant, we group edges by the transition they realize.

This distinction is central to the present paper. A vertex invariant describes local morphology at a single partition. A jump invariant describes how that morphology changes under one elementary transfer.

\subsection{Strict gradient orientations}

This construction is an orientation of selected edges, not a vector field. We use the term gradient orientation only because the orientation points in the direction of increasing values of a chosen invariant.

Let
\[
F:\Par(n)\to\mathbb R
\]
be a real-valued vertex invariant. The strict \(F\)-gradient orientation of \(G_n\) is obtained by orienting an edge \(\{\lambda,\mu\}\) from \(\lambda\) to \(\mu\) whenever
\[
F(\mu)>F(\lambda).
\]
Edges satisfying
\[
F(\lambda)=F(\mu)
\]
are called \(F\)-plateau edges.

This is not a vector field; it is an orientation of the graph induced by the scalar invariant \(F\). We use the term gradient orientation because the edges point in the direction of increasing \(F\).

Thus every oriented edge is of one of three types with respect to \(F\):
\[
\Delta_eF>0,\qquad \Delta_eF<0,\qquad \Delta_eF=0.
\]

\begin{proposition}[Acyclicity of strict gradient orientations]
Let \(F:\Par(n)\to\mathbb R\) be a real-valued vertex invariant. The strict \(F\)-gradient orientation of \(G_n\) has no directed cycles.
\end{proposition}

\begin{proof}
Along every directed edge, the value of \(F\) strictly increases. Therefore along any directed path, \(F\) strictly increases at each step. A directed cycle would require \(F\) to return to its initial value after a strict increase, which is impossible.
\end{proof}

This gives a precise sense in which a vertex invariant may define a gradient-type structure on \(G_n\). The important point is that such a structure is invariant-relative: different invariants may define different, and sometimes opposed, directions of increase.

The stronger language of monotone paths and corridors will be developed later in the paper.

\section{Support jumps}\label{sec:support-jumps}

The support size
\[
\sigma(\lambda)=|\supp(\lambda)|
\]
is one of the simplest vertex invariants of a partition. Its edgewise behavior already illustrates the basic mechanism studied in this paper: even a single elementary transfer may change the local type of a vertex.

In this section we prove that support jumps are universally bounded and that they are completely controlled by local multiplicity data.

\subsection{Elementary transfer notation}

Let
\[
\lambda=(1^{m_1}2^{m_2}3^{m_3}\cdots)
\]
be a partition of \(n\). An elementary transfer may be described as follows.

Choose a donor part of size \(p\ge 1\). Choose a recipient part of size \(q\ge 0\), where \(q=0\) is allowed and represents the creation of a new part. The transfer replaces
\[
p\mapsto p-1
\]
and
\[
q\mapsto q+1.
\]
A part of size \(0\) is omitted. If \(q>0\), the recipient is an existing part of size \(q\). If \(p=q\), the donor and recipient rows are required to be distinct, so that \(m_p(\lambda)\ge 2\).

Such a transfer defines an edge of \(G_n\) only when the resulting partition is distinct from the original one. Thus identity moves, such as the case \(q=p-1\) where the two affected row sizes are merely interchanged, do not define edges. The case \(p=1,q=0\) is also excluded for the same reason.

Equivalently, for a legal elementary transfer, the multiplicities change only at the sizes
\[
p,\qquad q,\qquad p-1,\qquad q+1,
\]
with the convention that size \(0\) is not part of the support, and the resulting partition \(\mu\) satisfies \(\mu\ne\lambda\). These expressions are to be understood as a set of affected sizes; some of them may coincide. Then \(\lambda\sim\mu\) in \(G_n\).

\subsection{The support-jump bound}

\begin{proposition}[Support jumps are bounded]
Let \(\lambda\sim\mu\) be an edge of \(G_n\). Then
\[
|\sigma(\mu)-\sigma(\lambda)|\le 2.
\]
Equivalently, for every oriented edge \(e=(\lambda,\mu)\),
\[
\Delta_e\sigma\in\{-2,-1,0,1,2\}.
\]
\end{proposition}

\begin{proof}
Let \(\mu\) be obtained from \(\lambda\) by an elementary transfer
\[
p\mapsto p-1,\qquad q\mapsto q+1.
\]
Only the selected donor and recipient rows are changed. Consequently, only the original sizes \(p\) and \(q\), if \(q>0\), can lose their last occurrence, and only the resulting sizes \(p-1\), if \(p>1\), and \(q+1\) can become newly present.

The support size can therefore decrease by at most \(2\) and increase by at most \(2\). Hence
\[
-2\le \sigma(\mu)-\sigma(\lambda)\le 2.
\]
This proves the assertion.
\end{proof}

The bound is sharp. For example, in \(G_8\) the transfer
\[
(4,4)\longrightarrow (4,3,1)
\]
is obtained by moving one unit from one part of size \(4\) to a new part. Here \(p=4\) and \(q=0\), so the recipient is a newly created part. The support changes from
\[
\{4\}
\]
to
\[
\{4,3,1\},
\]
so
\[
\Delta\sigma=2.
\]
The reverse edge has support jump \(-2\).

\subsection{Support jumps as local multiplicity data}

The previous proposition gives a universal bound. We now record the more precise fact that the support jump is determined by the transfer data together with the local multiplicities of the sizes affected by the transfer.

Let
\[
\lambda=(1^{m_1}2^{m_2}3^{m_3}\cdots)
\]
and suppose that \(\mu\) is obtained from \(\lambda\) by the transfer
\[
p\mapsto p-1,\qquad q\mapsto q+1.
\]

Let \(m_i=m_i(\lambda)\). Define the multiplicity change
\[
\varepsilon_i=m_i(\mu)-m_i(\lambda).
\]
Then \(\varepsilon_i=0\) unless
\[
i\in\{p,q,p-1,q+1\},
\]
again ignoring the index \(0\).

The support jump can therefore be written as
\[
\Delta\sigma=
\sum_{i\ge 1}
\left(
\mathbf 1_{m_i+\varepsilon_i>0}-\mathbf 1_{m_i>0}
\right).
\]
Since all terms vanish outside the finite set
\[
\{p,q,p-1,q+1\}\cap \mathbb Z_{\ge 1},
\]
the support jump is determined by the transfer data \((p,q)\) together with the local multiplicity data at those sizes.

\begin{proposition}[Local multiplicity description of support jumps]
Let \(\mu\) be obtained from \(\lambda\) by an elementary transfer
\[
p\mapsto p-1,\qquad q\mapsto q+1.
\]
Then
\[
\Delta\sigma=\sigma(\mu)-\sigma(\lambda)
\]
is completely determined by the transfer data \((p,q)\) together with the multiplicities of the sizes
\[
p,\qquad q,\qquad p-1,\qquad q+1
\]
in \(\lambda\), with the convention that size \(0\) is omitted.

More explicitly, let \(r\) be the number of support sizes that disappear during the transfer, and let \(a\) be the number of support sizes that are newly introduced during the transfer. Then
\[
\Delta\sigma=a-r.
\]
\end{proposition}

\begin{proof}
All multiplicities outside the set
\[
\{p,q,p-1,q+1\}
\]
remain unchanged and therefore do not affect the difference
\[
\sigma(\mu)-\sigma(\lambda).
\]

A support size disappears precisely when its multiplicity is positive in \(\lambda\) and becomes zero in \(\mu\). Let the number of such sizes be \(r\). A support size is newly introduced precisely when its multiplicity is zero in \(\lambda\) and positive in \(\mu\). Let the number of such sizes be \(a\).

All other affected sizes remain in the support both before and after the transfer, or remain absent both before and after the transfer. Therefore
\[
\sigma(\mu)=\sigma(\lambda)-r+a,
\]
and hence
\[
\Delta\sigma=a-r.
\]

Since disappearance and appearance can occur only among the affected sizes, the value of \(\Delta\sigma\) is determined by the transfer data \((p,q)\) together with the multiplicities of
\[
p,\qquad q,\qquad p-1,\qquad q+1
\]
in \(\lambda\). This proves the proposition.
\end{proof}

\subsection{Support-active and support-neutral edges}

The support-jump bound gives a small finite classification of edges with respect to support size.

An oriented edge \(e=(\lambda,\mu)\) is called \emph{support-increasing} if
\[
\Delta_e\sigma>0,
\]
\emph{support-decreasing} if
\[
\Delta_e\sigma<0,
\]
and \emph{support-neutral} if
\[
\Delta_e\sigma=0.
\]

The strongest possible support-increasing edges are those with
\[
\Delta_e\sigma=2,
\]
and the strongest possible support-decreasing edges are those with
\[
\Delta_e\sigma=-2.
\]

By the preceding propositions, every support-active edge is governed by local multiplicity data at the affected sizes. Thus support activity is a strictly local phenomenon in the multiplicity profile of the partition.

\subsection{Consequences for jump signatures}

Since the third component of the basic jump signature
\[
J(e)=(\Delta_ed,\Delta_e\delta,\Delta_e\sigma)
\]
is the support jump, we immediately obtain the following constraint.

\begin{corollary}[Universal support constraint]\label{cor:universal-support-constraint}
For every \(n\) and every oriented edge \(e\) of \(G_n\), if
\[
J(e)=(u,v,w),
\]
then
\[
w\in\{-2,-1,0,1,2\}.
\]
Equivalently,
\[
|w|\le 2.
\]
\end{corollary}

\begin{proof}
This is exactly the support-jump bound applied to the third component of \(J(e)\).
\end{proof}

Thus the support component of the jump signature is uniformly bounded, independently of \(n\). This is in contrast with the degree and local-dimension components, whose behavior is governed by larger local neighbor and simplex systems and will be discussed in the next section.

\section{Degree and local-dimension jumps}\label{sec:degree-dimension-jumps}

The support jump has a universal local bound because an elementary transfer affects only two original and two resulting sizes. Degree and local-dimension jumps are subtler: they are still attached to a single edge, but their magnitudes are governed by larger local structures: the neighbor system of a vertex and the local simplex structure in the clique complex.

In this section we introduce degree and local-dimension jumps as edgewise invariants and record the basic structural decompositions that will be used later in the transition taxonomy and in the computational atlas.

\subsection{Degree jumps and neighbor-system variation}

For a partition \(\lambda\in\Par(n)\), let
\[
N(\lambda)=\{\nu\in\Par(n):\nu\sim\lambda\}
\]
be the neighbor set of \(\lambda\) in \(G_n\). Then
\[
d(\lambda)=|N(\lambda)|.
\]

For an oriented edge \(e=(\lambda,\mu)\), the degree jump is
\[
\Delta_ed=d(\mu)-d(\lambda)=|N(\mu)|-|N(\lambda)|.
\]

The degree jump therefore measures the net change in the local neighbor system when one passes from \(\lambda\) to \(\mu\).

\begin{proposition}[Degree jump as neighbor-system imbalance]
Let \(e=(\lambda,\mu)\) be an oriented edge of \(G_n\). Then
\[
\Delta_ed=|N(\mu)\setminus N(\lambda)|-|N(\lambda)\setminus N(\mu)|.
\]
In particular,
\[
|\Delta_ed|\le |N(\lambda)\triangle N(\mu)|,
\]
where \(N(\lambda)\triangle N(\mu)\) denotes the symmetric difference of the two neighbor sets.
\end{proposition}

\begin{proof}
We have the disjoint decompositions
\[
N(\lambda)=\bigl(N(\lambda)\cap N(\mu)\bigr)\sqcup\bigl(N(\lambda)\setminus N(\mu)\bigr)
\]
and
\[
N(\mu)=\bigl(N(\lambda)\cap N(\mu)\bigr)\sqcup\bigl(N(\mu)\setminus N(\lambda)\bigr).
\]
Therefore
\[
|N(\mu)|-|N(\lambda)|=|N(\mu)\setminus N(\lambda)|-|N(\lambda)\setminus N(\mu)|.
\]
The inequality follows immediately from
\[
\bigl||N(\mu)\setminus N(\lambda)|-|N(\lambda)\setminus N(\mu)|\bigr|
\le |N(\mu)\setminus N(\lambda)|+|N(\lambda)\setminus N(\mu)|.
\]
The right-hand side is precisely
\[
|N(\lambda)\triangle N(\mu)|.
\]
\end{proof}

Although elementary, this identity is conceptually useful. It shows that a degree jump is not only a comparison of two numbers. It is the net result of neighbors gained and neighbors lost under an elementary transition.

\subsection{Degree-birth and degree-death sets}

For an oriented edge \(e=(\lambda,\mu)\), define the degree-birth set by
\[
B_d(e)=N(\mu)\setminus N(\lambda)
\]
and the degree-death set by
\[
D_d(e)=N(\lambda)\setminus N(\mu).
\]

Thus \(B_d(e)\) consists of neighbors present at \(\mu\) but not at \(\lambda\), while \(D_d(e)\) consists of neighbors present at \(\lambda\) but not at \(\mu\). With this notation,
\[
\Delta_ed=|B_d(e)|-|D_d(e)|.
\]

We also define the degree-reorganization size of \(e\) by
\[
\rho_d(e)=|B_d(e)|+|D_d(e)|.
\]
Equivalently,
\[
\rho_d(e)=|N(\lambda)\triangle N(\mu)|.
\]

The quantity \(\Delta_ed\) measures the net degree change, while \(\rho_d(e)\) measures the total size of the neighbor-system reorganization. These two quantities should not be confused. It may happen that
\[
\Delta_ed=0
\]
while
\[
\rho_d(e)>0.
\]
This occurs already in \(G_4\); see Example~\ref{ex:g4-neutral-active}. In that case the edge is degree-neutral numerically, but the local neighbor system still changes.

This distinction will be important later when we separate numerical jump signatures from fuller local transition data.

\subsection{Transfer-system formulation}

The neighbor set of \(\lambda\) may also be described in terms of legal elementary transfers.

Let \(\mathcal T(\lambda)\) be the set of distinct partitions different from \(\lambda\) that are obtained from \(\lambda\) by one legal elementary transfer. Then
\[
N(\lambda)=\mathcal T(\lambda),
\]
and hence
\[
d(\lambda)=|\mathcal T(\lambda)|.
\]

For an oriented edge \(e=(\lambda,\mu)\), we therefore have
\[
\Delta_ed=|\mathcal T(\mu)|-|\mathcal T(\lambda)|.
\]

This formulation emphasizes that a degree jump is produced by a change in the local transfer system. Some transfer outputs available at \(\lambda\) disappear at \(\mu\), and other transfer outputs become available only after the move.

Thus
\[
\Delta_ed=
\#\{\text{new transfer outputs at }\mu\}
-
\#\{\text{lost transfer outputs from }\lambda\},
\]
where both transfer-output sets are regarded as subsets of \(\Par(n)\). Here ``new'' means elements of \(N(\mu)\setminus N(\lambda)\), and ``lost'' means elements of \(N(\lambda)\setminus N(\mu)\).

\subsection{Local-dimension jumps}

Recall that
\[
K_n=\Cl(G_n)
\]
is the clique complex of the partition graph. For a vertex \(\lambda\), the local simplex dimension is
\[
\delta(\lambda)=\dim_{\mathrm{loc}}(\lambda),
\]
that is, the maximum dimension of a simplex of \(K_n\) containing \(\lambda\).

For an oriented edge \(e=(\lambda,\mu)\), the local-dimension jump is
\[
\Delta_e\delta=\delta(\mu)-\delta(\lambda).
\]

Unlike support jumps, local-dimension jumps are controlled not merely by the appearance or disappearance of part sizes, but by the change in maximal clique structures containing the endpoint.

Let
\[
\mathcal S(\lambda)
\]
denote the family of simplices of \(K_n\) containing \(\lambda\). Then
\[
\delta(\lambda)=\max_{\Sigma\in\mathcal S(\lambda)}\dim\Sigma.
\]
Thus
\[
\Delta_e\delta=
\max_{\Sigma\in\mathcal S(\mu)}\dim\Sigma
-
\max_{\Sigma\in\mathcal S(\lambda)}\dim\Sigma.
\]

In this sense, a local-dimension jump records whether the elementary transfer enters, exits, or remains within a region of higher local simplicial thickness.

\subsection{Simplex-layer boundaries}

For each integer \(r\ge 0\), define the local simplex layer
\[
L_{\ge r}(n)=\{\lambda\in\Par(n):\delta(\lambda)\ge r\}.
\]

The edge boundary of this layer is
\[
\partial_E L_{\ge r}(n)=\{\{\lambda,\mu\}\in E(G_n):\lambda\in L_{\ge r}(n),\ \mu\notin L_{\ge r}(n)\}.
\]
Equivalently, \(\partial_E L_{\ge r}(n)\) consists of the edges across which the predicate
\[
\delta(\lambda)\ge r
\]
changes truth value.

\begin{proposition}[Dimension jumps detect simplex-layer crossings]
Let \(e=(\lambda,\mu)\) be an oriented edge of \(G_n\).

If
\[
\lambda\notin L_{\ge r}(n)
\quad\text{and}\quad
\mu\in L_{\ge r}(n),
\]
then
\[
\Delta_e\delta>0.
\]

If
\[
\lambda\in L_{\ge r}(n)
\quad\text{and}\quad
\mu\notin L_{\ge r}(n),
\]
then
\[
\Delta_e\delta<0.
\]
Every crossing of the edge boundary of \(L_{\ge r}(n)\) is detected by a nonzero local-dimension jump.
\end{proposition}

\begin{proof}
If \(\lambda\notin L_{\ge r}(n)\), then
\[
\delta(\lambda)<r.
\]
If \(\mu\in L_{\ge r}(n)\), then
\[
\delta(\mu)\ge r.
\]
Hence
\[
\delta(\mu)-\delta(\lambda)>0.
\]
The exit case is identical with the roles of \(\lambda\) and \(\mu\) reversed.
\end{proof}

The converse is true after allowing the threshold \(r\) to vary.

\begin{corollary}[Every nonzero dimension jump crosses a simplex-layer boundary]
Let \(e=(\lambda,\mu)\) be an oriented edge. If
\[
\Delta_e\delta\neq 0,
\]
then the underlying unoriented edge belongs to
\[
\partial_E L_{\ge r}(n)
\]
for at least one integer \(r\).

More precisely, if
\[
\delta(\lambda)<\delta(\mu),
\]
then any integer \(r\) satisfying
\[
\delta(\lambda)<r\le \delta(\mu)
\]
has this property. If
\[
\delta(\mu)<\delta(\lambda),
\]
then any integer \(r\) satisfying
\[
\delta(\mu)<r\le \delta(\lambda)
\]
has this property.
\end{corollary}

\begin{proof}
Assume first that \(\delta(\lambda)<\delta(\mu)\). If
\[
\delta(\lambda)<r\le \delta(\mu),
\]
then
\[
\lambda\notin L_{\ge r}(n)
\quad\text{and}\quad
\mu\in L_{\ge r}(n).
\]
Hence the edge crosses the boundary of \(L_{\ge r}(n)\). The case \(\delta(\mu)<\delta(\lambda)\) is the same with the two endpoints interchanged.
\end{proof}

\subsection{Degree layers}

The same threshold-layer language applies to degree. For each integer \(r\ge 0\), define
\[
D_{\ge r}(n)=\{\lambda\in\Par(n):d(\lambda)\ge r\}.
\]
The edge boundary
\[
\partial_E D_{\ge r}(n)
\]
consists of edges whose endpoints lie on opposite sides of the degree threshold \(r\).

\begin{proposition}[Degree jumps and degree-layer boundaries]
Let \(e=(\lambda,\mu)\) be an oriented edge.

If the underlying edge crosses the boundary of \(D_{\ge r}(n)\), then
\[
\Delta_ed\neq 0.
\]

Conversely, if
\[
\Delta_ed\neq 0,
\]
then the underlying unoriented edge crosses the boundary of \(D_{\ge r}(n)\) for at least one integer \(r\).
\end{proposition}

\begin{proof}
This is the same threshold argument as in the local-dimension case, applied to the integer-valued invariant \(d\).

If the edge crosses the boundary of \(D_{\ge r}(n)\), then one endpoint has degree at least \(r\) and the other has degree less than \(r\). Therefore their degrees are unequal, so \(\Delta_ed\neq0\).

Conversely, if \(d(\lambda)\neq d(\mu)\), assume without loss of generality that
\[
d(\lambda)<d(\mu).
\]
Then for every integer \(r\) satisfying
\[
d(\lambda)<r\le d(\mu),
\]
the edge crosses the boundary of \(D_{\ge r}(n)\).
\end{proof}

\subsection{Why degree and local dimension are different from support}

The preceding discussion reveals an important contrast.

For support size, we have the universal edgewise bound
\[
|\Delta_e\sigma|\le 2.
\]
This follows directly from the fact that one elementary transfer changes at most two old row sizes and at most two new row sizes.

For degree and local simplex dimension, the same argument does not apply. A single elementary transfer changes only a small part of the partition, but it may reorganize several possible transfers or alter the maximal clique structures through the endpoint. Therefore degree and local-dimension jumps must be studied through the local transfer system, the local simplex structure, and computation.

In summary:
\[
\begin{array}{c|c}
\text{Invariant} & \text{Edgewise behavior} \\
\hline
\sigma & \text{universally bounded support jump} \\
d & \text{neighbor-system reorganization} \\
\delta & \text{maximal local simplex reorganization}
\end{array}
\]

This distinction is one of the guiding points of the paper. Support jumps provide a small exact model of edgewise transition. Degree and local-dimension jumps provide transition data that must be organized by signatures, layer boundaries, and computational atlases.

\section{Joint jump profiles and transition types}\label{sec:joint-profiles}

The previous sections introduced the three basic jump components
\[
\Delta_ed,\qquad \Delta_e\delta,\qquad \Delta_e\sigma.
\]
Taken separately, these jumps measure different forms of edgewise variation. The degree jump measures variation in local graph branching, the local-dimension jump measures variation in local simplicial thickness, and the support jump measures variation in the number of distinct part sizes.

In this section we study these quantities jointly. The main object is the jump signature
\[
J(e)=(\Delta_ed,\Delta_e\delta,\Delta_e\sigma).
\]
We do not attempt here to classify all realized signatures for all \(n\); that is a computational and asymptotic problem. Instead, we introduce a stable language for describing edgewise transition types.

\subsection{Signed and absolute jump profiles}

We now use the jump signature \(J(e)\) and absolute jump signature \(|J|(e)\) introduced in Section~\ref{subsec:jump-signatures}. The signed profile records direction, while the absolute profile records magnitude. If the orientation of an edge \(e\) is reversed, then
\[
J(\bar e)=-J(e),
\]
whereas
\[
|J|(\bar e)=|J|(e).
\]

\begin{proposition}[Signed symmetry of jump profiles]
Let \(e=(\lambda,\mu)\) be an oriented edge and let \(\bar e=(\mu,\lambda)\). Then
\[
J(\bar e)=-J(e).
\]
In particular,
\[
|J|(\bar e)=|J|(e).
\]
\end{proposition}

\begin{proof}
Each component of \(J(e)\) is the signed jump of a vertex invariant. Each signed jump changes sign under reversal of orientation. Taking absolute values removes this dependence on orientation.
\end{proof}

\subsection{Active and neutral components}

For an oriented edge \(e\), define its active component set by
\[
A(e)=\{F\in\{d,\delta,\sigma\}:\Delta_eF\neq0\}.
\]
Equivalently, \(A(e)\) records which entries of the vector
\[
J(e)=(\Delta_ed,\Delta_e\delta,\Delta_e\sigma)
\]
are nonzero.

We say that \(e\) is:
\[
\begin{array}{ll}
\textit{degree-active} & \text{if } \Delta_ed\neq0,\\[2mm]
\textit{dimension-active} & \text{if } \Delta_e\delta\neq0,\\[2mm]
\textit{support-active} & \text{if } \Delta_e\sigma\neq0.
\end{array}
\]
Similarly, \(e\) is:
\[
\begin{array}{ll}
\textit{degree-neutral} & \text{if } \Delta_ed=0,\\[2mm]
\textit{dimension-neutral} & \text{if } \Delta_e\delta=0,\\[2mm]
\textit{support-neutral} & \text{if } \Delta_e\sigma=0.
\end{array}
\]

Finally, \(e\) is called \emph{fully neutral} if
\[
J(e)=(0,0,0).
\]

Since reversing the orientation of \(e\) changes the signs of all jump components but not their vanishing, the active component set is independent of orientation:
\[
A(\bar e)=A(e).
\]

\subsection{Pure and mixed transitions}

The active component set gives a first taxonomy of edgewise transitions.

We say that an edge \(e\) is a \emph{pure transition} if exactly one of the three basic jump components is active:
\[
|A(e)|=1.
\]
Thus there are three pure transition types:
\[
\text{pure degree},\qquad \text{pure dimension},\qquad \text{pure support}.
\]

We say that \(e\) is a \emph{mixed transition} if at least two components are active:
\[
|A(e)|\ge2.
\]
A mixed transition may be degree-dimension mixed, degree-support mixed, dimension-support mixed, or fully mixed.

The edge is \emph{fully mixed} if
\[
\Delta_ed\neq0,\qquad \Delta_e\delta\neq0,\qquad \Delta_e\sigma\neq0.
\]

This terminology separates two questions that should not be confused:
\begin{enumerate}
\item how large the jumps are;
\item how many morphological layers change at once.
\end{enumerate}
An edge may have small jump magnitude but still be mixed. Conversely, an edge may have a large degree jump while being pure degree-active.

\subsection{Transition rank}

\begin{definition}[Transition rank]
The transition rank of an oriented edge \(e\) is
\[
\trk(e)=|A(e)|.
\]
Thus
\[
\trk(e)\in\{0,1,2,3\}.
\]
\end{definition}

The four possibilities are:
\[
\begin{array}{c|l}
\trk(e) & \text{meaning}\\
\hline
0 & \text{fully neutral}\\
1 & \text{pure transition}\\
2 & \text{two-component mixed transition}\\
3 & \text{fully mixed transition}
\end{array}
\]

Since \(A(e)\) depends only on whether jump components vanish, the transition rank is independent of reversing the orientation of \(e\). Therefore \(\trk\) also defines an invariant of the underlying unoriented edge; we write \(\trk(\{\lambda,\mu\})\) for the common value of the two orientations.

\subsection{Sign patterns}

When an orientation has been chosen, one may refine the active component set by signs.

For a signed jump profile
\[
J(e)=(u,v,w),
\]
define the sign pattern
\[
\sgn J(e)=(\sgn u,\sgn v,\sgn w),
\]
where each entry belongs to
\[
\{-1,0,+1\}.
\]

The sign pattern records whether the edge moves upward, downward, or along a plateau for each of the three invariants.

For example,
\[
(+,+,+)
\]
means that degree, local dimension, and support size all increase along the chosen orientation. The pattern
\[
(+,-,0)
\]
means that degree increases, local dimension decreases, and support size remains constant.

Such sign patterns are especially useful when comparing gradient orientations induced by different invariants. An edge may be upward for degree and downward for local dimension. Thus gradient behavior is not intrinsic to the edge alone; it is relative to the invariant or direction being used.

\begin{definition}[Coherent and opposed mixed transitions]
Let \(e\) be a mixed transition. We say that \(e\) is \emph{sign-coherent} if all nonzero components of \(J(e)\) have the same sign.

We say that \(e\) is \emph{sign-opposed} if among the nonzero components of \(J(e)\) there are both positive and negative entries.
\end{definition}

For example,
\[
(+,+,0)
\]
is sign-coherent, while
\[
(+,-,0)
\]
is sign-opposed.

This distinction uses a chosen orientation. However, the property of being sign-coherent or sign-opposed is preserved under reversing the edge orientation, since all nonzero signs change simultaneously.

\subsection{Layer-crossing interpretation}

Each integer-valued invariant determines a family of threshold layers. Let
\[
F:\Par(n)\to\mathbb Z
\]
be an integer-valued vertex invariant. Define
\[
L_F^{\ge r}(n)=\{\lambda\in\Par(n):F(\lambda)\ge r\}.
\]

We say that an unoriented edge \(\{\lambda,\mu\}\) crosses the boundary of \(L_F^{\ge r}(n)\) if exactly one of its endpoints belongs to \(L_F^{\ge r}(n)\). Equivalently, the two endpoints lie on opposite sides of the threshold \(r\).

\begin{proposition}[Nonzero jumps are threshold crossings]
Let \(F:\Par(n)\to\mathbb Z\) be an integer-valued vertex invariant, and let \(e=(\lambda,\mu)\) be an oriented edge.

If
\[
\Delta_eF\neq0,
\]
then the underlying unoriented edge crosses the boundary of \(L_F^{\ge r}(n)\) for at least one integer \(r\).

More precisely, if
\[
F(\lambda)<F(\mu),
\]
then the edge crosses every threshold \(r\) satisfying
\[
F(\lambda)<r\le F(\mu).
\]
If
\[
F(\mu)<F(\lambda),
\]
then the edge crosses every threshold \(r\) satisfying
\[
F(\mu)<r\le F(\lambda).
\]

Conversely, if an edge crosses the boundary of \(L_F^{\ge r}(n)\) for some \(r\), then
\[
\Delta_eF\neq0.
\]
\end{proposition}

\begin{proof}
Suppose first that \(F(\lambda)<F(\mu)\). Then for every integer \(r\) with
\[
F(\lambda)<r\le F(\mu),
\]
we have
\[
\lambda\notin L_F^{\ge r}(n),\qquad \mu\in L_F^{\ge r}(n).
\]
Hence the edge crosses the boundary of that threshold layer.

The case \(F(\mu)<F(\lambda)\) is identical after interchanging the two endpoints.

Conversely, if an edge crosses the boundary of \(L_F^{\ge r}(n)\), then one endpoint has \(F\)-value at least \(r\) and the other has \(F\)-value less than \(r\). Hence the two \(F\)-values are not equal, so \(\Delta_eF\neq0\).
\end{proof}

Applying this proposition to the three invariants
\[
d,\qquad \delta,\qquad \sigma
\]
shows that the active component set \(A(e)\) has a layer-crossing interpretation.

An edge is degree-active exactly when it crosses at least one degree threshold layer. It is dimension-active exactly when it crosses at least one local simplex threshold layer. It is support-active exactly when it crosses at least one support-size threshold layer.

Thus the transition rank
\[
\trk(e)
\]
counts how many of the three basic threshold-layer systems are crossed by the edge.

\subsection{Jump magnitude and crossed thresholds}

The previous proposition can be sharpened slightly. The absolute jump magnitude counts the number of integer thresholds crossed.

\begin{proposition}[Jump magnitude counts crossed thresholds]
Let \(F:\Par(n)\to\mathbb Z\) be an integer-valued invariant. For an oriented edge \(e=(\lambda,\mu)\), the number of integer thresholds \(r\) such that the underlying unoriented edge crosses the boundary of
\[
L_F^{\ge r}(n)
\]
is exactly
\[
|\Delta_eF|.
\]
\end{proposition}

\begin{proof}
Assume without loss of generality that
\[
F(\lambda)<F(\mu).
\]
The crossed thresholds are precisely the integers \(r\) satisfying
\[
F(\lambda)<r\le F(\mu).
\]
There are exactly
\[
F(\mu)-F(\lambda)=|\Delta_eF|
\]
such integers. The opposite case is identical.
\end{proof}

This observation will be useful in the computational atlas. A large jump is not merely a large numerical difference; it is an edge that crosses many integer threshold layers at once.

\subsection{Numerical neutrality versus structural neutrality}

The jump signature
\[
J(e)
\]
is a numerical invariant of an edge. It records the change in selected numerical invariants. It does not record the full local structure around the endpoints.

Therefore one must distinguish
\[
J(e)=(0,0,0)
\]
from genuine structural sameness.

A fully neutral edge may still connect vertices with different neighbor sets, different local transfer systems, or different positions in the graph.

\begin{example}[A fully neutral but structurally active edge in \(G_4\)]\label{ex:g4-neutral-active}
Consider the graph \(G_4\). The partitions of \(4\) are
\[
(4),\quad (3,1),\quad (2,2),\quad (2,1,1),\quad (1,1,1,1).
\]
The vertices
\[
(3,1)\quad\text{and}\quad (2,1,1)
\]
are adjacent.

Both have degree \(3\):
\[
d(3,1)=3,\qquad d(2,1,1)=3.
\]
Both have support size \(2\):
\[
\sigma(3,1)=2,\qquad \sigma(2,1,1)=2.
\]
Both lie in the triangle
\[
(3,1),\quad (2,2),\quad (2,1,1),
\]
so their local simplex dimension is \(2\):
\[
\delta(3,1)=2,\qquad \delta(2,1,1)=2.
\]
Hence the edge
\[
(3,1)\sim(2,1,1)
\]
has
\[
J(e)=(0,0,0).
\]

However, the neighbor sets are not the same:
\[
N(3,1)=\{(4),(2,2),(2,1,1)\},
\]
whereas
\[
N(2,1,1)=\{(3,1),(2,2),(1,1,1,1)\}.
\]
Thus the edge is fully neutral with respect to the basic numerical jump signature, but it is not structurally neutral at the level of neighbor systems.
\end{example}

This example shows why jump signatures should be treated as numerical summaries of edgewise morphology, not as complete invariants of local transition structure.

\subsection{Admissible, forbidden, and realized signatures}

For fixed \(n\), the set of realized signed jump signatures is
\[
\mathcal J(n)=\{J(e): e\text{ is an oriented edge of }G_n\}.
\]
The corresponding absolute signature set is
\[
\mathcal J_{\mathrm{abs}}(n)=\{|J|(e): e\text{ is an edge of }G_n\}.
\]

By the support-jump bound,
\[
\Delta_e\sigma\in\{-2,-1,0,1,2\}.
\]
Therefore no realized signature can have third component outside this range.

By Corollary~\ref{cor:universal-support-constraint}, every realized signature \(J(e)=(u,v,w)\) satisfies
\[
w\in\{-2,-1,0,1,2\},
\]
or equivalently \(|w|\le2\).

For the first two components, namely degree and local dimension, no analogous universal bound follows from the support argument alone. Their realized ranges must be studied using degree theory, local simplex structure, and computation.

\subsection{The basic transition taxonomy}

The preceding definitions give the following basic taxonomy:
\[
\begin{array}{c|c}
\text{type} & \text{condition}\\
\hline
\text{fully neutral} & (0,0,0)\\
\text{pure degree} & (\neq0,0,0)\\
\text{pure dimension} & (0,\neq0,0)\\
\text{pure support} & (0,0,\neq0)\\
\text{degree-dimension mixed} & (\neq0,\neq0,0)\\
\text{degree-support mixed} & (\neq0,0,\neq0)\\
\text{dimension-support mixed} & (0,\neq0,\neq0)\\
\text{fully mixed} & (\neq0,\neq0,\neq0)
\end{array}
\]

This taxonomy forgets the magnitudes and signs of the nonzero components. It records only which morphological layers are active.

A finer taxonomy is obtained by recording:
\begin{enumerate}
\item signs of nonzero components;
\item absolute magnitudes;
\item threshold layers crossed;
\item position of the edge in the global geometry of \(G_n\).
\end{enumerate}
Thus one has a hierarchy:
\[
\text{active set}\Rightarrow \text{sign pattern}\Rightarrow \text{absolute magnitude}\Rightarrow \text{geometric localization}.
\]

The computational part of the paper will use this hierarchy to organize edge atlases.

\subsection{Structural constraints and computational patterns}

At this stage, only some aspects of the transition taxonomy are established as elementary structural facts.

The following facts are used formally:
\begin{enumerate}
\item signed profiles change sign under orientation reversal;
\item absolute profiles are unoriented edge invariants;
\item the support component is universally bounded by \(2\);
\item nonzero jumps are exactly threshold-layer crossings for the corresponding invariant;
\item jump magnitude counts crossed integer thresholds;
\item transition rank counts the number of basic threshold-layer systems crossed by the edge.
\end{enumerate}

By contrast, the following questions are computational or conjectural at this stage:
\begin{enumerate}
\item which jump signatures are actually realized for large \(n\);
\item how the distribution of transition ranks evolves with \(n\);
\item whether fully mixed transitions concentrate in central or rear-central regions;
\item whether large degree jumps correlate with support jumps;
\item whether large local-dimension jumps concentrate near simplex-layer boundaries;
\item whether the absolute jump spectrum stabilizes in any normalized sense as \(n\to\infty\).
\end{enumerate}

This distinction is important. The edgewise language is exact; the observed morphology of realized signatures is empirical until proved.

\section{Gradient-type behavior and monotone corridors}\label{sec:gradient-corridors}

Jump invariants are edgewise. Gradient-type behavior is pathwise. A single edge records one elementary transition, while a path records a sequence of transitions and hence may exhibit monotonicity, drift, or oscillation.

There is no canonical gradient on \(G_n\) without specifying a vertex invariant. Instead, every real-valued invariant
\[
F:\Par(n)\to\mathbb R
\]
defines its own gradient-type structure. This section makes that statement precise and introduces monotone paths and corridors.

\subsection{Strict gradient orientation}

Let
\[
F:\Par(n)\to\mathbb R
\]
be a real-valued vertex invariant. The strict \(F\)-gradient orientation of \(G_n\) is obtained by orienting an edge
\[
\{\lambda,\mu\}
\]
from \(\lambda\) to \(\mu\) whenever
\[
F(\mu)>F(\lambda).
\]

If
\[
F(\lambda)=F(\mu),
\]
the edge is called an \(F\)-plateau edge.

Thus every oriented edge is of one of three types with respect to \(F\):
\[
\Delta_eF>0,\qquad \Delta_eF<0,\qquad \Delta_eF=0.
\]
Only the first type is oriented in the strict \(F\)-gradient direction.

\begin{proposition}[Strict gradient orientations are acyclic]
Let \(F:\Par(n)\to\mathbb R\). The strict \(F\)-gradient orientation of \(G_n\) has no directed cycles.
\end{proposition}

\begin{proof}
Along every directed edge, the value of \(F\) strictly increases. Therefore along any directed path, \(F\) strictly increases at each step. A directed cycle would require \(F\) to return to its initial value after a strict increase, which is impossible.
\end{proof}

\subsection{Monotone paths}

Let
\[
P=(\lambda_0,\lambda_1,\dots,\lambda_m)
\]
be a path in \(G_n\).

We say that \(P\) is \emph{strictly \(F\)-increasing} if
\[
F(\lambda_0)<F(\lambda_1)<\cdots<F(\lambda_m).
\]

We say that \(P\) is \emph{weakly \(F\)-increasing} if
\[
F(\lambda_0)\le F(\lambda_1)\le\cdots\le F(\lambda_m).
\]

Similarly, one defines strictly and weakly \(F\)-decreasing paths.

A strictly \(F\)-increasing path is precisely a directed path in the strict \(F\)-gradient orientation. A weakly \(F\)-increasing path may also include \(F\)-plateau edges.

\begin{proposition}[Length bound for strict monotone paths]
Let \(F:\Par(n)\to\mathbb R\), and let
\[
V_F(n)=F(\Par(n))
\]
be the set of values assumed by \(F\) on \(\Par(n)\). Every strictly \(F\)-increasing path in \(G_n\) has length at most
\[
|V_F(n)|-1.
\]
\end{proposition}

\begin{proof}
Along a strictly \(F\)-increasing path, the values of \(F\) at successive vertices are pairwise distinct. Therefore the path cannot contain more vertices than the number of distinct \(F\)-values. Hence its length is at most \(|V_F(n)|-1\).
\end{proof}

\subsection{Multi-invariant monotonicity}

For several invariants
\[
F_1,\dots,F_k:\Par(n)\to\mathbb R,
\]
one can impose simultaneous monotonicity conditions.

A path
\[
P=(\lambda_0,\dots,\lambda_m)
\]
is \emph{weakly \((F_1,\dots,F_k)\)-increasing} if, for every \(i\) and every \(j\),
\[
F_i(\lambda_j)\le F_i(\lambda_{j+1}).
\]

It is \emph{strictly \((F_1,\dots,F_k)\)-increasing} if each step is weakly increasing in all components and at least one component strictly increases at each step.

Equivalently, weak monotonicity means that the vector-valued invariant
\[
\mathbf F=(F_1,\dots,F_k)
\]
moves monotonically with respect to the product partial order on \(\mathbb R^k\). The strict version additionally requires a nonzero increase in at least one component at each step.

For the basic invariants of this paper, the main case is
\[
\mathbf F=(d,\delta,\sigma).
\]
A path is weakly \((d,\delta,\sigma)\)-increasing if degree, local dimension, and support size never decrease along the path.

\subsection{Compatibility and opposition of gradients}

The jump signature
\[
J(e)=(\Delta_ed,\Delta_e\delta,\Delta_e\sigma)
\]
shows immediately whether an edge is compatible with a chosen multi-invariant monotonicity condition.

For example, an oriented edge \(e\) is compatible with weak \((d,\delta,\sigma)\)-increase if
\[
\Delta_ed\ge0,\qquad \Delta_e\delta\ge0,\qquad \Delta_e\sigma\ge0.
\]
It is compatible with strict \((d,\delta,\sigma)\)-increase if, in addition,
\[
J(e)\neq(0,0,0).
\]

By contrast, an edge with sign pattern
\[
(+,-,0)
\]
is degree-increasing but dimension-decreasing. Such an edge is compatible with a degree-gradient path, but incompatible with simultaneous degree-dimension increase.

This illustrates a basic principle: gradient-type behavior is invariant-relative. An edge may be upward for one invariant and downward for another.

\subsection{Cone-compatible transitions}

Let
\[
\mathbf F=(F_1,\dots,F_k)
\]
be a vector-valued invariant, and let
\[
C\subseteq\mathbb R^k
\]
be a prescribed set of allowed jump vectors, typically a cone or sign region.

\begin{definition}[Gradient-compatible edge]
An oriented edge \(e\) is called \emph{\((\mathbf F,C)\)-compatible} if
\[
\Delta_e\mathbf F=(\Delta_eF_1,\dots,\Delta_eF_k)\in C.
\]
\end{definition}

For example, weak simultaneous increase corresponds to the cone
\[
C=\mathbb R_{\ge0}^k.
\]
Strict simultaneous increase corresponds to
\[
C=\mathbb R_{\ge0}^k\setminus\{0\}.
\]

For
\[
\mathbf F=(d,\delta,\sigma),
\]
one can prescribe sign cones such as
\[
C_{+++}=\mathbb R_{\ge0}^3,
\]
\[
C_{++*}=\{(x,y,z):x\ge0,\ y\ge0\},
\]
or
\[
C_{+-*}=\{(x,y,z):x\ge0,\ y\le0\}.
\]

The first cone describes simultaneous weak increase of degree, local dimension, and support size. The second describes simultaneous weak increase of degree and local dimension, with no condition imposed on support. The third describes degree increase together with local-dimension decrease.

This notation permits a uniform treatment of monotone, mixed, and opposed gradient behavior.

\subsection{Corridors}

We use the term corridor for a controlled pathwise pattern of compatible edge transitions, not for an additional structure imposed on \(G_n\).

The formal definition used here is path-level.

\begin{definition}[Path corridor]
Let
\[
\mathbf F=(F_1,\dots,F_k)
\]
be a vector-valued invariant and let \(C\subseteq\mathbb R^k\) be a prescribed set of allowed jump vectors. A path
\[
P=(\lambda_0,\dots,\lambda_m)
\]
is called a \emph{\((\mathbf F,C)\)-corridor} if every oriented edge
\[
(\lambda_j,\lambda_{j+1})
\]
is \((\mathbf F,C)\)-compatible.
\end{definition}

Thus a degree-increasing corridor is a path along which
\[
d(\lambda_0)\le d(\lambda_1)\le\cdots\le d(\lambda_m).
\]
A local-dimension-increasing corridor is a path along which
\[
\delta(\lambda_0)\le \delta(\lambda_1)\le\cdots\le \delta(\lambda_m).
\]
A support-increasing corridor is a path along which
\[
\sigma(\lambda_0)\le \sigma(\lambda_1)\le\cdots\le \sigma(\lambda_m).
\]

One may also speak about subgraph-level corridors, but this requires a stronger additional convention than mere weak monotonicity. Indeed, for a scalar invariant \(F\), any subgraph can have its edges oriented from lower \(F\)-value to higher \(F\)-value, with plateau edges oriented arbitrarily; therefore a naive weak subgraph corridor would be essentially vacuous.

For this reason, the present paper uses only path corridors as a formal notion. Subgraph-level corridor language is not used formally here; it may be introduced in later computational work after additional conventions have been specified.

\subsection{Plateau edges and plateau regions}

For an invariant \(F\), an edge \(e\) is an \(F\)-plateau edge if
\[
\Delta_eF=0.
\]

The graph whose edge set consists of all \(F\)-plateau edges records the regions in which \(F\) is locally constant along elementary transfers.

Plateau behavior is important because weak monotone paths may move through plateau regions before entering a strictly increasing region.

For the three basic invariants, we get degree plateaus, local-dimension plateaus, and support-size plateaus. Their intersections are also meaningful. A fully neutral edge
\[
J(e)=(0,0,0)
\]
is simultaneously a plateau edge for all three invariants. However, as seen in the previous section, such an edge need not be structurally inactive.

\subsection{Gradient orientation and layer boundaries}

The \(F\)-gradient viewpoint is closely related to threshold layers.

For an integer-valued invariant \(F\), recall that
\[
L_F^{\ge r}(n)=\{\lambda\in\Par(n):F(\lambda)\ge r\}.
\]

A strict \(F\)-gradient edge crosses at least one threshold-layer boundary in the upward direction. Conversely, if an oriented edge crosses a threshold-layer boundary from outside to inside, then its \(F\)-jump is positive.

\begin{proposition}[Strict gradient edges are upward threshold crossings]
Let \(F:\Par(n)\to\mathbb Z\), and let \(e=(\lambda,\mu)\) be an oriented edge. Then
\[
\Delta_eF>0
\]
if and only if there exists an integer \(r\) such that
\[
\lambda\notin L_F^{\ge r}(n),\qquad \mu\in L_F^{\ge r}(n).
\]
More precisely, this holds for every integer \(r\) satisfying
\[
F(\lambda)<r\le F(\mu).
\]
\end{proposition}

\begin{proof}
If \(\Delta_eF>0\), then \(F(\lambda)<F(\mu)\), and every integer \(r\) satisfying
\[
F(\lambda)<r\le F(\mu)
\]
has the stated property.

Conversely, if there exists \(r\) such that
\[
\lambda\notin L_F^{\ge r}(n),\qquad \mu\in L_F^{\ge r}(n),
\]
then
\[
F(\lambda)<r\le F(\mu),
\]
and hence \(F(\lambda)<F(\mu)\). Therefore \(\Delta_eF>0\).
\end{proof}

\subsection{Mixed-gradient corridors}

The most interesting corridors in \(G_n\) may not be monotone for all invariants at once. Instead, one may study paths of prescribed mixed behavior.

For example, the condition
\[
\Delta d\ge0,\qquad \Delta\delta\ge0,\qquad \Delta\sigma\ \text{unrestricted}
\]
describes paths along which graph branching and local simplicial thickness do not decrease, while support size is allowed to fluctuate.

Another possible condition is
\[
\Delta d\ge0,\qquad \Delta\delta\le0.
\]
Such a path moves upward in degree while moving downward in local dimension.

Such behavior is possible because degree and local simplicial thickness measure different aspects of local morphology.

The cone notation introduced above allows these cases to be treated uniformly. For
\[
\mathbf F=(d,\delta,\sigma),
\]
one can prescribe a cone or sign region
\[
C\subseteq\mathbb R^3
\]
and study paths whose jump signatures remain in \(C\).

\subsection{Axial and positional drift}

Some gradient language depends not on \(d,\delta,\sigma\), but on positional invariants.

Let
\[
a(\lambda)=\lambda_1
\]
and
\[
b(\lambda)=\ell(\lambda).
\]
The quantity
\[
\alpha(\lambda)=a(\lambda)-b(\lambda)
\]
measures signed displacement from the self-conjugate axis in the \((a,b)\)-projection, while
\[
\adist(\lambda)=|\alpha(\lambda)|
\]
measures axial distance in this coarse sense.

An oriented edge \(e=(\lambda,\mu)\) is called:
\[
\begin{array}{ll}
\textit{axis-inward} & \text{if } \adist(\mu)<\adist(\lambda),\\[2mm]
\textit{axis-outward} & \text{if } \adist(\mu)>\adist(\lambda),\\[2mm]
\textit{axis-neutral} & \text{if } \adist(\mu)=\adist(\lambda).
\end{array}
\]

This gives a positional gradient language distinct from degree, local dimension, and support.

An axis-inward edge need not increase degree or local dimension. Conversely, a degree-increasing or dimension-increasing edge need not move inward toward the self-conjugate axis. Thus axial drift and invariant-gradient drift should be compared, not identified.

\subsection{Formal and descriptive corridor language}

The following notions are formal definitions used in the paper:
\begin{enumerate}
\item strict \(F\)-gradient orientation;
\item \(F\)-plateau edge;
\item strictly and weakly \(F\)-monotone paths and path corridors;
\item vector-valued monotonicity with respect to a prescribed set \(C\) of allowed jump vectors;
\item upward threshold crossings;
\item axis-inward, axis-outward, and axis-neutral edges defined by \(|a-b|\).
\end{enumerate}

The following notions require computation or additional proof before they can be used as structural claims:
\begin{enumerate}
\item existence of long monotone corridors between specified regions;
\item uniqueness or canonicality of such corridors;
\item concentration of degree-increasing paths near the spine;
\item systematic inward drift of high-dimension edges;
\item relation between rear-central thickening and mixed-gradient paths;
\item asymptotic stabilization of corridor types as \(n\) grows.
\end{enumerate}

Thus the formal corridor language used here is path-level: a corridor is a path satisfying explicitly stated jump inequalities. Broader subgraph-level or statistical corridor notions are left for later computational work, where the additional conventions can be specified explicitly.

\section{Localization and computational atlas}\label{sec:atlas}

The preceding sections developed the formal language of edgewise jumps, jump signatures, transition ranks, gradient-compatible paths, and threshold-layer crossings. We now describe how this language can be used to build computational atlases of \(G_n\). No large-scale computations are performed in this paper; the section is a blueprint for future implementations. The constructions below are not meant to replace proof-level statements. Rather, they specify what data should be recorded in order to compare edgewise transition behavior across values of \(n\).

The purpose of such an atlas is not merely to count edges. Rather, it is to locate different transition types inside the global geometry of the partition graph. In particular, we want to know where large jumps occur, how pure and mixed transitions are distributed, and how edgewise behavior interacts with vertex-level structures such as degree layers, simplex layers, support strata, axial regions, the spine, shells, and boundary zones.

This section gives a reproducible framework for such computations. In this paper we include only a small illustrative computation. The same protocol can be applied to larger values of \(n\), with extended tables treated as supplementary material.

\subsection{The edge dataset for fixed \texorpdfstring{\(n\)}{n}}

For fixed \(n\), define the oriented edge dataset
\[
\mathcal E^{\mathrm{or}}_n=\{(\lambda,\mu):\lambda,\mu\in\Par(n),\ \lambda\sim\mu\}.
\]
Each unoriented edge gives two oriented edges.

For every oriented edge
\[
e=(\lambda,\mu),
\]
we record the endpoint values
\[
d(\lambda),\quad d(\mu),
\]
\[
\delta(\lambda),\quad \delta(\mu),
\]
\[
\sigma(\lambda),\quad \sigma(\mu),
\]
and hence the jump signature
\[
J(e)=(\Delta_ed,\Delta_e\delta,\Delta_e\sigma).
\]

For localization, we also record positional data for both endpoints:
\[
a(\lambda)=\lambda_1,\qquad b(\lambda)=\ell(\lambda),
\]
\[
a(\mu)=\mu_1,\qquad b(\mu)=\ell(\mu),
\]
as well as the coarse axial distances
\[
|a(\lambda)-b(\lambda)|,\qquad |a(\mu)-b(\mu)|.
\]

Depending on the focus of the computation, one may also record simplex-layer membership, degree-layer membership, support pattern, shell-depth, boundary status, or membership in other regions defined in the partition-graph program.

\subsection{Signed and absolute jump spectra}

For fixed \(n\), the signed jump spectrum of the basic signature is
\[
\mathcal J(n)=\{J(e):e\in\mathcal E^{\mathrm{or}}_n\}.
\]
The absolute jump spectrum is
\[
\mathcal J_{\mathrm{abs}}(n)=\{|J|(e):e\in E(G_n)\},
\]
where \(|J|(e)\) denotes the common absolute signature of the two orientations of the unoriented edge \(e\).

Here the signed spectrum is naturally computed on oriented edges, while the absolute spectrum is naturally computed on unoriented edges.

For each component one may also compute one-dimensional spectra:
\[
\mathcal J_d(n)=\{\Delta_ed:e\in\mathcal E^{\mathrm{or}}_n\},
\]
\[
\mathcal J_\delta(n)=\{\Delta_e\delta:e\in\mathcal E^{\mathrm{or}}_n\},
\]
\[
\mathcal J_\sigma(n)=\{\Delta_e\sigma:e\in\mathcal E^{\mathrm{or}}_n\}.
\]

By the support-jump bound,
\[
\mathcal J_\sigma(n)\subseteq\{-2,-1,0,1,2\}.
\]
The support-jump spectrum is therefore is universally bounded. The degree and local-dimension spectra must be measured and compared across values of \(n\).

\subsection{Basic jump-range tables}

The first level of the computational atlas records the ranges of the three jump components.

For each \(n\), one should compute:
\[
\min \Delta d,\qquad \max \Delta d,\qquad \max|\Delta d|,
\]
\[
\min \Delta\delta,\qquad \max \Delta\delta,\qquad \max|\Delta\delta|,
\]
\[
\min \Delta\sigma,\qquad \max \Delta\sigma,\qquad \max|\Delta\sigma|.
\]

For larger computations, a basic table has the following form:
\[
\begin{array}{c|c|c|c|c|c}
n & |\Par(n)| & |E(G_n)| & \max|\Delta d| & \max|\Delta\delta| & \max|\Delta\sigma| \\
\hline
 & & & & &
\end{array}
\]

The final column must never exceed \(2\). Therefore it also functions as a computational consistency check.

\subsection{Histograms}

For an integer-valued invariant
\[
F:\Par(n)\to\mathbb Z,
\]
define the signed jump histogram
\[
H_F^{(n)}(k)=\#\{e\in\mathcal E^{\mathrm{or}}_n:\Delta_eF=k\}.
\]

For absolute jumps, define
\[
H_{|F|}^{(n)}(k)=\#\{\{\lambda,\mu\}\in E(G_n):|F(\mu)-F(\lambda)|=k\}.
\]

The signed histogram is symmetric.

\begin{proposition}[Symmetry of signed jump histograms]
Let \(F:\Par(n)\to\mathbb Z\) be any integer-valued invariant. Then
\[
H_F^{(n)}(k)=H_F^{(n)}(-k)
\]
for all \(k\).
\end{proposition}

\begin{proof}
The map
\[
(\lambda,\mu)\mapsto(\mu,\lambda)
\]
is an involution on the set of oriented edges. It sends edges with
\[
\Delta F=k
\]
to edges with
\[
\Delta F=-k.
\]
Hence the two classes have the same cardinality.
\end{proof}

Signed histograms are useful for checks and for directed-gradient analyses. Absolute histograms are better suited to unoriented edge morphology.

\subsection{Transition-rank statistics}

For every oriented edge \(e\), recall the transition rank
\[
\trk(e)=\#\{F\in\{d,\delta,\sigma\}:\Delta_eF\neq0\}.
\]
Since transition rank is independent of orientation, we also regard it as an invariant of the underlying unoriented edge. It is therefore naturally counted on unoriented edges.

Define
\[
T_r(n)=\#\{e\in E(G_n):\trk(e)=r\}.
\]
Here
\[
r\in\{0,1,2,3\}.
\]
The normalized frequencies are
\[
\tau_r(n)=\frac{T_r(n)}{|E(G_n)|}.
\]

These numbers measure how often elementary transfers are neutral, pure, mixed, or fully mixed with respect to the basic jump signature.

A transition-rank table has the form:
\[
\begin{array}{c|c|c|c|c|c}
n & T_0(n) & T_1(n) & T_2(n) & T_3(n) & \text{dominant rank}\\
\hline
 & & & & &
\end{array}
\]

This table answers a basic edgewise morphological question: are most elementary transitions invisible to the basic signature, pure, or mixed?

As a minimal check, consider again the graph \(G_4\). Its five unoriented edges have transition-rank distribution
\[
T_0(4)=1,\qquad T_1(4)=0,\qquad T_2(4)=2,\qquad T_3(4)=2.
\]
Indeed, the edge \((3,1)\sim(2,1,1)\) is fully neutral, the two edges adjacent to the central triangle but not to the endpoints \((4)\) and \((1,1,1,1)\) are degree-support mixed, and the two boundary edges \((4)\sim(3,1)\) and \((2,1,1)\sim(1,1,1,1)\) are fully mixed. The same example illustrates that neutral transition rank does not imply equality of neighbor systems.

\subsection{Active-set distributions}

Using the taxonomy of Section~\ref{sec:joint-profiles}, decompose the edge set into the eight basic active-set classes:
\[
E_{000},
\]
\[
E_d,\qquad E_\delta,\qquad E_\sigma,
\]
\[
E_{d\delta},\qquad E_{d\sigma},\qquad E_{\delta\sigma},
\]
\[
E_{d\delta\sigma}.
\]

For example,
\[
E_{d\sigma}
=
\{\{\lambda,\mu\}\in E(G_n):
 d(\lambda)\ne d(\mu),
\delta(\lambda)=\delta(\mu),
\sigma(\lambda)\ne\sigma(\mu)\}.
\]

This decomposition forgets magnitudes and signs, but records which morphological layers are active.

An active-set table has the form:
\[
\begin{array}{c|c|c|c|c|c|c|c|c}
n
& E_{000}
& E_d
& E_\delta
& E_\sigma
& E_{d\delta}
& E_{d\sigma}
& E_{\delta\sigma}
& E_{d\delta\sigma}
\\
\hline
 & & & & & & & &
\end{array}
\]

This table is more informative than transition rank alone, because it distinguishes, for example, degree-support mixing from degree-dimension mixing.

\subsection{Joint signature distributions}

The full signed joint distribution is
\[
C_n(u,v,w)=\#\{e\in\mathcal E^{\mathrm{or}}_n:J(e)=(u,v,w)\}.
\]

For absolute signatures, define
\[
C_n^{\mathrm{abs}}(u,v,w)=\#\{e\in E(G_n):|J|(e)=(u,v,w)\}.
\]

Because the support component satisfies
\[
w\in\{-2,-1,0,1,2\},
\]
the third coordinate is uniformly bounded. This makes slices by support jump especially useful.

For example, one may study separately:
\[
\Delta\sigma=0,
\]
\[
|\Delta\sigma|=1,
\]
and
\[
|\Delta\sigma|=2.
\]

The slice
\[
\Delta\sigma=0
\]
captures transitions that preserve the number of support sizes, while
\[
|\Delta\sigma|=2
\]
captures the strongest possible support transitions.

\subsection{Localization by \texorpdfstring{\((a,b)\)}{(a,b)}-stacks}

Recall that
\[
a(\lambda)=\lambda_1,\qquad b(\lambda)=\ell(\lambda).
\]
The pair
\[
(a(\lambda),b(\lambda))
\]
places \(\lambda\) in a stack of partitions with fixed largest part and fixed number of parts.

For an edge
\[
e=\{\lambda,\mu\},
\]
one can localize it by the endpoint pairs
\[
(a(\lambda),b(\lambda)),\qquad (a(\mu),b(\mu)).
\]

For visualization, it is often useful to assign the edge to a coarse position. One possible convention is the midpoint
\[
\left(\frac{a(\lambda)+a(\mu)}2,\frac{b(\lambda)+b(\mu)}2\right).
\]

Other conventions are also possible, such as assigning the edge to the endpoint with smaller value of a chosen invariant, larger value of a chosen invariant, or smaller axial distance. The convention should be stated explicitly.

Let
\[
R\subseteq \mathbb R^2
\]
be a region in the \((a,b)\)-plane. Let \(E_R(n)\) be the set of edges whose chosen representative point lies in \(R\).

For an invariant \(F\), define the local average absolute jump by
\[
A_F(R;n)=\frac{1}{|E_R(n)|}\sum_{\{\lambda,\mu\}\in E_R(n)}|F(\mu)-F(\lambda)|,
\]
when \(E_R(n)\neq\varnothing\).

For the full basic signature, define the total local jump activity
\[
A_J(R;n)=\frac{1}{|E_R(n)|}\sum_{\{\lambda,\mu\}\in E_R(n)}\bigl(|d(\mu)-d(\lambda)|+|\delta(\mu)-\delta(\lambda)|+|\sigma(\mu)-\sigma(\lambda)|\bigr).
\]

This gives a coarse heatmap of edgewise activity in the \((a,b)\)-projection.

\subsection{Localization relative to the self-conjugate axis}

The self-conjugate axis is visible in the \((a,b)\)-projection through the diagonal
\[
a=b.
\]
Define the signed axial displacement
\[
\alpha(\lambda)=a(\lambda)-b(\lambda)
\]
and the coarse axial distance
\[
\adist(\lambda)=|\alpha(\lambda)|.
\]

For an oriented edge \(e=(\lambda,\mu)\), define
\[
\Delta_e\adist=\adist(\mu)-\adist(\lambda).
\]

Then \(e\) is:
\[
\begin{array}{ll}
\textit{axis-inward} & \text{if } \Delta_e\adist<0,\\[2mm]
\textit{axis-outward} & \text{if } \Delta_e\adist>0,\\[2mm]
\textit{axis-neutral} & \text{if } \Delta_e\adist=0.
\end{array}
\]

This allows one to compare jump signatures with axial drift. For example, one may ask whether large positive degree jumps are more often axis-inward than axis-outward, or whether local-dimension jumps concentrate near small axial distance.

Such statements are computational observations unless proved separately.

\subsection{Localization relative to simplex layers}

For each \(r\), recall the local simplex layer
\[
L_{\ge r}(n)=\{\lambda:\delta(\lambda)\ge r\}.
\]
Its edge boundary is
\[
\partial_E L_{\ge r}(n).
\]

One can compute
\[
B_r(n)=|\partial_E L_{\ge r}(n)|.
\]

More refined oriented data include:
\[
B_r^+(n)=\#\{(\lambda,\mu)\in\mathcal E_n^{\mathrm{or}}:\delta(\lambda)<r\le\delta(\mu)\},
\]
and
\[
B_r^-(n)=\#\{(\lambda,\mu)\in\mathcal E_n^{\mathrm{or}}:\delta(\mu)<r\le\delta(\lambda)\}.
\]

If both orientations are counted, then
\[
B_r^+(n)=B_r^-(n).
\]

Every nonzero local-dimension jump is a crossing of one or more simplex-layer boundaries. More precisely, if
\[
|\Delta_e\delta|=k,
\]
then the edge crosses exactly \(k\) integer threshold boundaries for \(\delta\).

Large local-dimension jumps therefore can be interpreted as simultaneous crossings of several simplex thresholds.

\subsection{Large-jump edges}

For an invariant \(F\), define the unoriented large-jump edge set
\[
E_{|F|}^{\ge r}(n)=\{\{\lambda,\mu\}\in E(G_n):|F(\mu)-F(\lambda)|\ge r\}.
\]

For the full basic signature, there are several useful choices. For an unoriented edge \(e=\{\lambda,\mu\}\), define the \(L^1\)-type total activity by
\[
\|J(e)\|_1=|d(\mu)-d(\lambda)|+|\delta(\mu)-\delta(\lambda)|+|\sigma(\mu)-\sigma(\lambda)|.
\]
Then set
\[
E_{J,1}^{\ge r}(n)=\{e\in E(G_n):\|J(e)\|_1\ge r\}.
\]

Alternatively, one may use the \(L^\infty\)-type activity
\[
\|J(e)\|_\infty=\max\{|d(\mu)-d(\lambda)|,|\delta(\mu)-\delta(\lambda)|,|\sigma(\mu)-\sigma(\lambda)|\},
\]
and define
\[
E_{J,\infty}^{\ge r}(n)=\{e\in E(G_n):\|J(e)\|_\infty\ge r\}.
\]

The two choices answer different questions. The \(L^1\)-quantity measures total jump activity across all components. The \(L^\infty\)-quantity detects whether at least one component is large.

The chosen norm should always be stated explicitly.

\subsection{Suggested atlas figures}

The computational atlas should contain a small number of carefully chosen figures rather than an excessive number of tables.

The following figure types are especially useful. In the figure descriptions below, expressions such as \(|\Delta d|\) denote the corresponding absolute jump on an unoriented edge.

\subsubsection*{Edge heatmap in the \((a,b)\)-plane}

Edges are placed by midpoint or by another stated convention. Color or intensity records one of:
\[
|\Delta d|,
\]
\[
|\Delta\delta|,
\]
\[
|\Delta\sigma|,
\]
or total activity
\[
|\Delta d|+|\Delta\delta|+|\Delta\sigma|.
\]

\subsubsection*{Transition-rank atlas}

Edges are colored by
\[
\trk(e)\in\{0,1,2,3\}.
\]
This shows where pure and mixed transitions occur.

\subsubsection*{Support-jump atlas}

Edges are colored by
\[
\Delta\sigma\in\{-2,-1,0,1,2\}
\]
or by
\[
|\Delta\sigma|\in\{0,1,2\}.
\]
This directly visualizes the support-jump theorem.

\subsubsection*{Simplex-layer boundary atlas}

For a chosen threshold \(r\), draw the edge boundary
\[
\partial_E L_{\ge r}(n).
\]
This shows where local-dimension transitions occur.

\subsubsection*{Large-jump atlas}

Show only edges satisfying a condition such as
\[
|\Delta d|\ge r,
\]
\[
|\Delta\delta|\ge r,
\]
or
\[
\|J(e)\|_\infty\ge r.
\]
This isolates strong transition zones.

\subsection{Computational questions}

The atlas should be guided by a small number of questions.

\begin{problem}[Support sharpness]
For which values of \(n\) do jumps
\[
|\Delta\sigma|=2
\]
first occur? Where are these edges located?
\end{problem}

\begin{problem}[Degree-jump growth]
How does
\[
\max_{e\in E(G_n)}|\Delta_ed|
\]
grow with \(n\)? Is the growth bounded, slowly growing, or governed by special partition shapes?
\end{problem}

\begin{problem}[Local-dimension jumps]
How does
\[
\max_{e\in E(G_n)}|\Delta_e\delta|
\]
grow with \(n\)? Are large local-dimension jumps concentrated near simplex-layer boundaries?
\end{problem}

\begin{problem}[Mixed transitions]
Do mixed transitions dominate for large \(n\), or do pure transitions remain frequent? Which mixed class is most common:
\[
d\delta,\qquad d\sigma,\qquad \delta\sigma,\qquad d\delta\sigma?
\]
\end{problem}

\begin{problem}[Axial localization]
Are high-activity edges concentrated near the self-conjugate axis, near the boundary, or in rear-central regions?
\end{problem}

\begin{problem}[Corridor evidence]
Do degree-increasing, dimension-increasing, or mixed-gradient corridors appear as coherent geometric features in the atlas?
\end{problem}

\subsection{Structural and atlas-level statements}

The following statements are elementary structural facts:
\begin{enumerate}
\item the definitions of jump spectra and jump classes;
\item the symmetry of signed histograms;
\item the support-jump bound;
\item the interpretation of jump magnitude as the number of crossed integer thresholds;
\item the decomposition of edges by transition rank and active component set.
\end{enumerate}

The following statements are atlas-level until separately proved:
\begin{enumerate}
\item large degree jumps concentrate in a particular geometric region;
\item large local-dimension jumps concentrate near simplex-layer interfaces;
\item support jumps correlate with degree jumps;
\item fully mixed transitions mark transition interfaces;
\item jump spectra stabilize in some asymptotic sense;
\item gradient corridors are canonical or persistent across \(n\).
\end{enumerate}

This distinction is essential. The atlas can reveal recurring motifs, but recurrence alone does not make a motif a theorem.

\section{Conclusions and open problems}\label{sec:conclusion}

The purpose of this paper has been to pass from vertex-level morphology to edgewise morphology in the partition graph \(G_n\).

Earlier invariants such as degree, local simplex dimension, and support size describe the local structure of individual vertices. Here we study how these invariants change under a single elementary transfer. The main object is therefore no longer only a vertex invariant
\[
F(\lambda),
\]
but its edgewise jump
\[
\Delta_eF=F(\mu)-F(\lambda)
\]
along an oriented edge
\[
e=(\lambda,\mu).
\]

This shift produces a language for transitions, drift, and interface structure inside \(G_n\).

\subsection{Edgewise morphology}

The central notion introduced here is the jump signature
\[
J(e)=(\Delta_ed,\Delta_e\delta,\Delta_e\sigma),
\]
where
\[
d(\lambda)=\deg_{G_n}(\lambda),
\]
\[
\delta(\lambda)=\dim_{\mathrm{loc}}(\lambda),
\]
and
\[
\sigma(\lambda)=|\supp(\lambda)|.
\]

This signature records, for a single elementary transfer, how graph branching, local simplicial thickness, and support size change simultaneously.

The corresponding absolute signature
\[
|J|(e)=(|\Delta_ed|,|\Delta_e\delta|,|\Delta_e\sigma|)
\]
is an invariant of the underlying unoriented edge.

This distinction between signed and absolute jumps is essential. Signed jumps are appropriate for gradient-type questions, while absolute jumps are appropriate for unoriented atlases of transition magnitude.

\subsection{Proven structural constraints}

The formal part of the paper consists of several elementary but useful structural facts.

First, signed jumps change sign under reversal of orientation, while absolute jumps do not.

Second, support jumps satisfy the universal bound
\[
|\Delta_e\sigma|\le 2.
\]
The support component of every jump signature therefore belongs to the fixed set
\[
\{-2,-1,0,1,2\},
\]
independently of \(n\).

Third, support jumps are determined by local multiplicity data. If an elementary transfer is written as
\[
p\mapsto p-1,\qquad q\mapsto q+1,
\]
then the support jump is determined by the transfer data \((p,q)\) together with the multiplicities of the affected sizes
\[
p,\qquad q,\qquad p-1,\qquad q+1,
\]
understood as a set in which some entries may coincide.

Fourth, for every integer-valued invariant \(F\), nonzero jumps are exactly crossings of threshold layers
\[
L_F^{\ge r}(n)=\{\lambda:F(\lambda)\ge r\}.
\]
Moreover,
\[
|\Delta_eF|
\]
counts the number of integer thresholds crossed by the edge.

Finally, every real-valued invariant \(F\) defines a strict \(F\)-gradient orientation, and this orientation is acyclic.

These facts provide the elementary structural foundation of the paper.

\subsection{Pure, mixed, and neutral transitions}

The jump signature leads to a finite qualitative taxonomy of edge transitions.

An edge may be fully neutral, pure degree-active, pure dimension-active, pure support-active, mixed in two components, or fully mixed in all three components.

Equivalently, the transition rank
\[
\trk(e)=\#\{F\in\{d,\delta,\sigma\}:\Delta_eF\neq0\}
\]
takes values in
\[
\{0,1,2,3\}.
\]

This rank counts how many of the basic morphological layer systems are crossed by the edge.

A fully neutral edge need not be structurally trivial. It may preserve the numerical values of \(d\), \(\delta\), and \(\sigma\), while still changing the neighbor system or other local data. The jump signature is therefore best understood as a controlled numerical shadow of edgewise morphology, not as a complete invariant of local transition structure.

\subsection{Gradient language and corridors}

The paper also introduced a cautious gradient-type language.

There is no single canonical gradient on \(G_n\). Instead, each invariant
\[
F:\Par(n)\to\mathbb R
\]
defines its own gradient orientation, monotone paths, plateau edges, and threshold crossings.

For several invariants
\[
\mathbf F=(F_1,\dots,F_k),
\]
one may impose monotonicity with respect to a prescribed set of allowed jump vectors
\[
C\subseteq\mathbb R^k
\]
by requiring
\[
\Delta_e\mathbf F\in C
\]
along each edge of a path.

This gives a flexible language for degree-increasing, dimension-increasing, support-increasing, mixed-gradient, or opposed-gradient corridors.

The main caution is that different invariants may point in different directions. An edge can be degree-increasing and dimension-decreasing, or support-increasing and degree-neutral. Gradient behavior is therefore always relative to the chosen invariant or vector of invariants.

\subsection{Computational atlas}

The formal language developed in this paper naturally leads to a computational atlas of edgewise transitions.

For fixed \(n\), one may compute
\[
\mathcal J(n)=\{J(e):e\in\mathcal E_n^{\mathrm{or}}\},
\]
\[
\mathcal J_{\mathrm{abs}}(n)=\{|J|(e):e\in E(G_n)\},
\]
as well as histograms of
\[
\Delta d,\qquad \Delta\delta,\qquad \Delta\sigma,
\]
transition-rank distributions, active-component distributions, and joint signature counts.

The atlas can then be localized in several ways:
\begin{enumerate}
\item by \((a,b)\)-stacks, where
\[
a(\lambda)=\lambda_1,\qquad b(\lambda)=\ell(\lambda);
\]
\item by axial distance
\[
|a(\lambda)-b(\lambda)|;
\]
\item by simplex-layer membership;
\item by degree layers;
\item by support strata;
\item by previously studied shell and boundary structures.
\end{enumerate}

This provides a framework for detecting where strong transitions occur and whether they concentrate near central, axial, rear-central, boundary, or shell-interface regions.

\subsection{Open problems}

The paper leaves several natural problems open.

\begin{problem}[Degree-jump growth]
Determine the growth of
\[
\max_{e\in E(G_n)}|\Delta_ed|
\]
as a function of \(n\).

Is this growth bounded, logarithmic, polynomial, or governed by special partition shapes?
\end{problem}

\begin{problem}[Local-dimension jump growth]
Determine the growth of
\[
\max_{e\in E(G_n)}|\Delta_e\delta|.
\]

Are large local-dimension jumps concentrated near boundaries of simplex layers, or do they also occur deep inside thick regions?
\end{problem}

\begin{problem}[Realized jump signatures]
Describe or classify the realized signature set
\[
\mathcal J(n)=\{J(e):e\in\mathcal E_n^{\mathrm{or}}\}.
\]

How does this set change with \(n\)? Does it stabilize in any normalized or asymptotic sense?
\end{problem}

\begin{problem}[Mixed transition dominance]
For large \(n\), which transition ranks dominate?

Do most edges become mixed, or do pure transitions remain significant?

More specifically, what are the asymptotic behaviors of
\[
T_0(n),\qquad T_1(n),\qquad T_2(n),\qquad T_3(n)?
\]
\end{problem}

\begin{problem}[Support jumps and degree jumps]
Support jumps are universally bounded, while degree jumps may be larger.

Is there a systematic relation between large degree jumps and nonzero support jumps?

Can large degree jumps occur frequently along support-neutral edges?
\end{problem}

\begin{problem}[Support jumps and local dimension]
Do support-active edges tend to cross local simplex layers?

Equivalently, is there a strong statistical or structural relation between
\[
\Delta\sigma\neq0
\]
and
\[
\Delta\delta\neq0?
\]
\end{problem}

\begin{problem}[Localization of large jumps]
Where do large jumps occur?

Are they concentrated near the self-conjugate axis, near the spine, near rear-central thickening, near boundary frameworks, near simplex-layer interfaces, or in several distinct regions?
\end{problem}

\begin{problem}[Gradient corridors]
Do there exist long monotone corridors for degree, local dimension, or support?

Can one characterize paths along which
\[
d,\qquad \delta,\qquad \sigma
\]
are simultaneously nondecreasing?

Are there natural opposed corridors, for example degree-increasing but dimension-decreasing paths?
\end{problem}

\begin{problem}[Edgewise reconstruction of morphology]
How much of the global morphology of \(G_n\) can be recovered from jump data alone?

For example, can central regions, shells, or high-thickness zones be detected purely from edgewise transition statistics?
\end{problem}

\begin{problem}[Relation to global topology]
The clique complex
\[
K_n=\Cl(G_n)
\]
has comparatively simple global homotopy type in the established theory, while the local morphology of \(G_n\) is highly structured.

Can jump-type invariants help explain this contrast?

More specifically, can edgewise transition data clarify how locally complex regions are assembled into a globally simple topological object?
\end{problem}

\subsection{Toward a synthesis}

The edgewise language developed here is intended to serve as a bridge between two parts of the partition-graph program.

On one side are vertex-level theories: degree, support, local simplex dimension, shells, layers, boundary regions, and axial morphology.

On the other side is a broader synthetic question:
\begin{quote}
How can the partition graph exhibit rich local and regional morphology while its associated clique complex remains globally topologically simple?
\end{quote}

Jump invariants and gradient-type structures do not answer this question by themselves. They do, however, provide a language for studying how local regimes meet, how transitions occur, and how transition patterns are organized across \(G_n\).

In this sense, edgewise morphology is a natural intermediate layer between static local invariants and global topological structure.

\section*{Acknowledgements}

The author acknowledges the use of ChatGPT (OpenAI) for discussion, structural planning, and editorial assistance during the preparation of this manuscript. All mathematical statements, proofs, computations, and final wording were checked and approved by the author, who takes full responsibility for the contents of the paper.


{\small
\begin{thebibliography}{99}
\setlength{\itemsep}{0pt}

\bibitem{Andrews}
G. E. Andrews,
\textit{The Theory of Partitions},
Cambridge Mathematical Library,
Cambridge University Press, Cambridge, 1998.

\bibitem{Diestel}
R. Diestel,
\textit{Graph Theory},
5th ed., Graduate Texts in Mathematics, vol. 173,
Springer, Berlin, 2017.

\bibitem{Kozlov}
D. Kozlov,
\textit{Combinatorial Algebraic Topology},
Algorithms and Computation in Mathematics, vol. 21,
Springer, Berlin, 2008.

\bibitem{Macdonald}
I. G. Macdonald,
\textit{Symmetric Functions and Hall Polynomials},
2nd ed., Oxford Mathematical Monographs,
Oxford University Press, Oxford, 1995.

\bibitem{Stanley}
R. P. Stanley,
\textit{Enumerative Combinatorics. Vol. 1},
2nd ed., Cambridge Studies in Advanced Mathematics, vol. 49,
Cambridge University Press, Cambridge, 2012.

\bibitem{Lyu2026Homotopy}
F. B. Lyudogovskiy,
\textit{The homotopy type of the clique complex of the partition graph},
arXiv:2603.14370v1 [math.CO], 2026.

\bibitem{Lyu2026Local}
F. B. Lyudogovskiy,
\textit{Local Morphology of the Partition Graph},
arXiv:2603.18696v1 [math.CO], 2026.

\bibitem{Lyu2026Growing}
F. B. Lyudogovskiy,
\textit{The Partition Graph as a Growing Discrete Geometric Object},
arXiv:2603.21221v1 [math.CO], 2026.

\bibitem{Lyu2026Directional}
F. B. Lyudogovskiy,
\textit{Directional Geometry and Anisotropy in the Partition Graph},
arXiv:2603.25488v1 [math.CO], 2026.

\bibitem{Lyu2026DegreeTheory}
F. B. Lyudogovskiy,
\textit{Degree theory of the partition graph: exact maxima, profiles, and fibres},
arXiv:2603.27248v1 [math.CO], 2026.

\bibitem{Lyu2026SimplicialShells}
F. B. Lyudogovskiy,
\textit{Simplicial shells and thickness in the partition graph},
arXiv:2603.28171v1 [math.CO], 2026.

\bibitem{Lyu2026SimplexLayers}
F. B. Lyudogovskiy,
\textit{Simplex Layers and Phase Boundaries in the Partition Graph},
arXiv:2603.29988v1 [math.CO], 2026.

\bibitem{Lyu2026Support}
F. B. Lyudogovskiy,
\textit{Support and Support Jumps in the Partition Graph},
arXiv:2604.11837v1 [math.CO], 2026.

\end{thebibliography}
}
\end{document}